\documentclass[12pt]{article}

\usepackage{graphicx}%
\usepackage{multirow}%
\usepackage{amsmath,amssymb,amsfonts}%
\usepackage{amsthm}%
\usepackage{mathrsfs}%
\usepackage[title]{appendix}%
\usepackage{xcolor}%
\usepackage{textcomp}%
\usepackage{manyfoot}%
\usepackage{booktabs}%
\usepackage{algorithm}%
\usepackage{algorithmicx}%
\usepackage{algpseudocode}%
\usepackage{listings}%
\usepackage{setspace}
\usepackage{url}
\usepackage{hyperref}

\usepackage{lipsum}

\definecolor{red}{rgb}{1,0,0}
\definecolor{blue}{rgb}{.2,.2,.8}

\usepackage{tikz}

\newtheorem{theorem}{Theorem}[section]
\newtheorem{corollary}[theorem]{Corollary}

\newtheorem{conjecture}{Conjecture}

\newtheorem{lemma}[theorem]{Lemma}

\theoremstyle{definition}

\DeclareMathOperator{\ord}{ord}

\begin{document}
\allowdisplaybreaks

\title{The two-block odd partition function}

\author{
    Mircea Merca
	\\ 
	\footnotesize Department of Mathematical Methods and Models\\ 
		\footnotesize Fundamental Sciences Applied in Engineering Research Center\\ \footnotesize National University of Science and Technology Politehnica  Bucharest\\
\footnotesize RO-060042 Bucharest, Romania\\
        \footnotesize Academy of Romanian Scientists, RO-050044 Bucharest, Romania\\
	\footnotesize mircea.merca@upb.ro
	}
	\date{}
	\maketitle

\begin{abstract}
We investigate arithmetic and combinatorial properties of the function $a(n)$, which is the signed number of partitions of $n$ into exactly two distinct part sizes, each occurring an odd number of times. We prove that $a(n)\ge 0$ for all $n$, establishing a positivity phenomenon for a signed partition function arising from a double Lambert series. Furthermore, we show that $a(n)$ satisfies nontrivial divisibility properties modulo $3$. The results reveal unexpected arithmetic regularity in a family of partition functions defined by odd multiplicity constraints and restricted support.
\end{abstract}

 \smallskip




\section{Introduction}

Partitions with restricted multiplicities have long played a central role in combinatorics and number theory, with deep connections to $q$-series, modular forms, and arithmetic congruences. Classical examples include partitions into distinct parts, partitions into odd parts, and various refinements governed by modular and theta–function identities. Such families frequently exhibit unexpected structural and arithmetic phenomena, including positivity properties and Ramanujan-type congruences.

\smallskip
While restrictions on multiplicities have been extensively studied, the interaction between \textit{constraints on the support} (i.e., the number of distinct part sizes) and \textit{parity conditions on multiplicities} remains comparatively less explored. In particular, partition functions defined simultaneously by a fixed number of distinct parts and parity restrictions often display subtle cancellation effects, making their arithmetic behavior difficult to predict.

\smallskip
In this paper we investigate a natural example of such a function. We define $a(n)$, which we call the \emph{two-block odd partition function}, to be the signed number of partitions of $n$ into exactly two distinct part sizes, each occurring an odd number of times. More precisely, $a(n)$ counts partitions of the form
\[
n = (2r_1+1)x + (2r_2+1)y, 
\qquad 1 \le x < y,
\]
with sign $(-1)^{r_1+r_2}$. 

\smallskip
The associated generating function can be written as a double Lambert series, reflecting the interaction between parity constraints and the restriction to two distinct part sizes:
\[
A(q):=
\sum_{n \ge 1} a(n)\, q^n
=
\sum_{1 \le n_1 < n_2}
\frac{q^{n_1}}{1+q^{2n_1}}
\frac{q^{n_2}}{1+q^{2n_2}}
\]
Each factor enforces odd multiplicities through the expansion
\[
\frac{q^n}{1+q^{2n}}
=
q^n\bigl(1 - q^{2n} + q^{4n} - \cdots \bigr),
\]
and the alternating signs reflect the number of ``$2n$-packages'' added to the minimal configuration. Recent investigations on double Lambert series can be found in \cite{AAB}.

\smallskip
Although the definition of $a(n)$ involves only elementary parity
restrictions, its arithmetic behavior is far from obvious. The signs in the
generating function suggest substantial cancellation, and there is no
apparent reason to expect positivity of the resulting coefficients.

\smallskip
The novelty of the present work lies not in the use of classical divisor-sum
or sums-of-squares identities themselves, but in the discovery that the
signed two-block partition model admits an explicit cancellation-free
arithmetic description. This leads to an unexpected positivity result,
a factorization of the generating function into a classical partition
generating function and a theta-type correction series, and a Ramanujan-type
congruence modulo $3$.

\smallskip
To describe the coefficients of $A(q)$ explicitly, let
\[
\chi_4(n)=
\begin{cases}
\pm 1, & n\equiv \pm1 \pmod{4},\\[3pt]
0, & n\equiv 0,2 \pmod{4}.
\end{cases}
\]
denote the Dirichlet character modulo $4$, and define
\[
\lambda(n):=\sum_{d\mid n}\chi_4(d),\qquad
\sigma(n):=\sum_{d\mid n} d.
\]

The coefficients of $a(n)$ admit a remarkably simple arithmetic
description. The following theorem provides an explicit formula in terms
of the classical divisor-sum function and the character $\chi_4$. 
This explicit coefficient formula constitutes the first main result 
of the paper and underlies both the positivity property of $a(n)$ and 
the later arithmetic results.

\medskip
\begin{theorem}\label{thm:explicit}
	For every positive integer $n$,
	\[
	a(n)
	=
	\frac{\sigma(n)-\lambda(n)}{4}
	-\frac{\sigma(n/2)}{2}
	+\sigma(n/4),
	\]
	where, by convention, $\sigma(x)=0$ whenever
	$x$ is not a positive integer.
\end{theorem}
\medskip

An immediate consequence of Theorem~\ref{thm:explicit} is the
nonnegativity of the coefficients $a(n)$, despite the signed nature of
their combinatorial definition.

\medskip
\begin{corollary}\label{cor:positivity}
	For every positive integer $n$,
	\[
	a(n)\ge 0,
	\]
	with equality if and only if $n=1,2$.
\end{corollary}
\medskip

Our second result provides an unexpected connection with the family of
partitions in which odd parts are distinct while even parts are unrestricted.
Recall that the generating function of such partitions is
\[
P(q)=\sum_{n=0}^\infty \mathrm{pod}(n)\,q^n=\frac{(-q;q^2)_\infty}{(q^2;q^2)_\infty}=\frac{(q^2;q^2)_\infty}{(q;q)_\infty\, (q^4;q^4)_\infty},
\]
where for $|q|<1$, the $q$-Pochhammer symbol is defined by
\[
(a;q)_\infty := \prod_{n=0}^{\infty} (1-a\,q^n).
\]
Taking into account the triangular numbers $T_n=n(n+1)/2$, we obtain the following representation.

\medskip
\begin{theorem}\label{thm:factorization}
	The generating function of $a(n)$ admits the factorization
	$$A(q)=P(q)\,S(q),$$
	where
	\[
	S(q)=
	\sum_{r=1}^\infty
	(-1)^{T_{r-1}}\,
	T_{\lceil r/2\rceil}\,
	q^{T_{r+1}}.
	\]
\end{theorem}
\medskip

This representation separates the arithmetic structure of $A(q)$ into a
classical infinite product, corresponding to partitions with distinct odd
parts and unrestricted even parts, and a theta-type correction
series involving quadratic exponents and alternating signs. The quadratic
powers $q^{T_{r+1}}$ suggest a truncated theta-type structure, while
the triangular number $T_{\lceil r/2 \rceil }$ encodes a secondary
combinatorial refinement (see \cite[A008805]{Sloane}).

\smallskip
Theorem \ref{thm:factorization} yields the following pair of mutually inverse relations. The first corollary expresses 
$a(n)$ in terms of $pod(n)$.

\medskip
\begin{corollary}
	For every positive integer $n$,
	$$a(n)=\sum_{r=1}^\infty (-1)^{T_{r-1}}\, T_{\lceil r/2 \rceil}\, \mathrm{pod}\left(n-T_{r+1}\right)
	$$
\end{corollary}
\medskip

The second corollary yields a linear recurrence for 
$a(n)$, arising from a well-known theta identity \cite[p. 23]{Andrews76} satisfied by the reciprocal of the generating function of 
$\mathrm{pod}(n)$.

\medskip
\begin{corollary}
	For every positive integer $n$,
	$$
	\sum_{r=0}^\infty (-1)^{T_r}\, a(n-T_r)
	= \begin{cases}
	(-1)^{T_{k-1}}\, T_{\lceil k/2 \rceil}, &\text{if $n=T_{k+1}$},\\ 
	0, & \text{otherwise.}
	\end{cases}
	$$  
\end{corollary}
\medskip

According to Corollary \ref{cor:positivity} and Theorem~\ref{thm:factorization}, we observe that the generating function  
\[
\frac{(q^2;q^4)_\infty}{(q;q)_\infty}  \sum_{r=1}^\infty
(-1)^{T_{r-1}}\,
T_{\lceil r/2\rceil}\,
q^{T_{r+1}}
\]
is a power series in \(q\) with nonnegative coefficients. Extensive computational evidence suggests the following conjecture.

\medskip
\begin{conjecture}
	If \(N \equiv 1,2 \pmod{4}\), then the truncated series
	\[
	\frac{(q^2;q^4)_\infty}{(q;q)_\infty}  \sum_{r=1}^{N}
	(-1)^{T_{r-1}}\,
	T_{\lceil r/2\rceil}\,
	q^{T_{r+1}}
	\]
	has nonnegative coefficients. Moreover, for all \(n \ge 3\), the coefficient of \(q^n\) is positive.
\end{conjecture}
\medskip

Our third main result concerns the arithmetic of $a(n)$. We provide 
a Ramanujan-type congruence for this family of partition functions. 

\medskip
\begin{theorem}\label{Th3}
	For all integers $n \ge 0$, one has
	$$a(12n+11)\equiv 0 \pmod 3.$$
\end{theorem}
\medskip

The results reveal that even highly constrained partition functions, defined by simultaneous restrictions on the support and multiplicities of parts, can exhibit unexpectedly rigid arithmetic behavior. 

\smallskip
The remainder of the paper is organized as follows. In Section \ref{S2} we prove Theorem \ref{thm:explicit} and its corollary using classical identities for divisor sums and representations by sums of squares. Section \ref{S3} is devoted to the proof of Theorem \ref{thm:factorization}. In Section \ref{S4} we establish the congruence result stated in Theorem \ref{Th3} by embedding the generating function into the theory of modular forms and applying Sturm's theorem.

\section{Explicit Formula and Positivity}
\label{S2}

We aim to provide a proof that is as elementary as possible, relying only on classical identities for divisor sums and representations as sums of squares. 
We begin with the identity
\begin{align*}
\sum_{n\ge 1} \frac{q^n}{1+q^{2n}}
&= \sum_{n\ge 1} \frac{q^n(1-q^{2n})}{1-q^{4n}} \\
&= \sum_{n\ge 1} \frac{q^n}{1-q^{4n}}
- \sum_{n\ge 1} \frac{q^{3n}}{1-q^{4n}}.
\end{align*}
Expanding into geometric series, we obtain
\[
\sum_{n\ge 1} \frac{q^n}{1+q^{2n}}
= \sum_{n\ge 1} \lambda(n)\,q^n.
\]

The generating function of $a(n)$ can therefore be written as
\begin{align*}
\sum_{n\ge 1} a(n)\,q^n
&= \frac{1}{2}\left(\sum_{n\ge 1}\frac{q^n}{1+q^{2n}}\right)^2
-\frac{1}{2}\sum_{n\ge 1}\frac{q^{2n}}{(1+q^{2n})^2} \\
&= \frac{1}{2}\left(\sum_{n\ge 1}\lambda(n)\,q^n\right)^2
-\frac{1}{2}\sum_{n\ge 1}\bigl(\sigma(n)-4\sigma(n/2)\bigr)q^{2n},
\end{align*}
with the convention that $\sigma(x)=0$ if $x$ is not a positive integer.

It follows that $a(1)=0$, and for $n>1$,
\[
a(n)=\frac{1}{2}\left(
\sum_{k=1}^{n-1}\lambda(k)\lambda(n-k)
-\sigma(n/2)+4\sigma(n/4)
\right).
\]

Let
\[
r_2(m)=\#\{(a,b)\in\mathbb{Z}^2:\ a^2+b^2=m\}.
\]
By a classical theorem of Jacobi (see, for example, \cite[Theorem 278]{HardyWright}),
\[
r_2(m)=4\sum_{d\mid m}\chi_4(d)=4\lambda(m).
\]
Similarly, if
\[
r_4(m)=\#\{(a,b,c,d)\in\mathbb{Z}^4:\ a^2+b^2+c^2+d^2=m\},
\]
then (see \cite[Theorem 285]{HardyWright})
\[
r_4(m)=8\sum_{\substack{d\mid m\\4\nmid d}} d
=8\bigl(\sigma(m)-4\sigma(m/4)\bigr).
\]

On the other hand,
\begin{align*}
r_4(n)
&=\sum_{k=0}^{n} r_2(k)r_2(n-k) \\
&=\sum_{k=1}^{n-1} r_2(k)r_2(n-k)
+2r_2(0)r_2(n) \\
&=16\sum_{k=1}^{n-1}\lambda(k)\lambda(n-k)
+8\lambda(n).
\end{align*}
Hence,
\begin{align*}
\sum_{k=1}^{n-1}\lambda(k)\lambda(n-k)
&=\frac{1}{16}r_4(n)-\frac{\lambda(n)}{2} \\
&=\frac{1}{2}\bigl(\sigma(n)-4\sigma(n/4)-\lambda(n)\bigr).
\end{align*}
Substituting into the expression for $a(n)$ gives
\[
a(n)
=\frac{\sigma(n)-\lambda(n)}{4}
-\frac{\sigma(n/2)}{2}
+\sigma(n/4).
\]
In particular, $a(1)=a(2)=0$. This completes the proof of Theorem \ref{thm:explicit}.

\medskip
To prove positivity for $n>2$, note that
\[
\sigma(n)>\sum_{d\mid n}1\ge \lambda(n),
\]
so that $\sigma(n)-\lambda(n)>0$ for $n>1$.

If $n\ge 3$ is odd, then $\sigma(n/2)=\sigma(n/4)=0$, and therefore
\[
a(n)=\frac{\sigma(n)-\lambda(n)}{4}>0.
\]

Now write $n=2^t m$ with $m$ odd. Then
\[
\sigma(2^t m)=(2^{t+1}-1)\sigma(m),
\qquad
\lambda(2^t m)=\lambda(m).
\]

If $t=1$, then
\begin{align*}
a(n)
&=\frac{\sigma(2m)-\lambda(2m)}{4}
-\frac{\sigma(m)}{2} \\
&=\frac{3\sigma(m)-\lambda(m)}{4}
-\frac{\sigma(m)}{2}
=\frac{\sigma(m)-\lambda(m)}{4}>0.
\end{align*}

If $t>1$, a straightforward computation yields
\begin{align*}
a(n)
&=\frac{\sigma(m)-\lambda(m)}{4}
+(2^{t-1}-1)\sigma(m)>0.
\end{align*}

Thus $a(n)>0$ for all $n>2$, which completes the proof of Corollary \ref{cor:positivity}.

\section{Factorization of the Generating Function}
\label{S3}

We begin by considering the Euler-type product
\begin{align}
A(z,q):=\prod_{n\ge1}
\left(
1+z\frac{q^n}{1+q^{2n}}
\right).
\label{eq3.1}
\end{align}

Expanding $A(z,q)$ as a power series in $z$, the coefficient of $z^2$
records partitions involving exactly two distinct part sizes. Hence
\[
A(q)=[z^2]A(z,q).
\]

For each $n\ge1$ we write
\[
1+z\frac{q^n}{1+q^{2n}}
=
\frac{1+q^{2n}+zq^n}{1+q^{2n}}.
\]

Let $\alpha$ and $\beta$ be the roots of
\[
t^2-zt+1=0.
\]

Then
\[
\alpha+\beta=z,
\qquad
\alpha\beta=1,
\qquad
\beta=\alpha^{-1},
\]
and therefore
\[
1+q^{2n}+zq^n
=
(1+\alpha q^n)(1+\beta q^n).
\]

Substituting into \eqref{eq3.1}, we obtain
\[
A(z,q)
=
\prod_{n\ge1}
\frac{(1+\alpha q^n)(1+\beta q^n)}
{1+q^{2n}}
=
\frac{(-\alpha q;q)_\infty
	(-\beta q;q)_\infty}
{(-q^2;q^2)_\infty}.
\]

Since $\beta=\alpha^{-1}$, we have
\[
A(z,q)
=
\frac{(-\alpha q;q)_\infty
	(-q/\alpha;q)_\infty}
{(-q^2;q^2)_\infty}.
\]

Applying the Jacobi triple product identity \cite[Theorem 44]{Johnson}
\[
\sum_{m\in\mathbb Z}
q^{m(m-1)/2}x^m
=
(q;q)_\infty
(-x;q)_\infty
(-q/x;q)_\infty
\]
with $x=\alpha$, we obtain
\[
(-\alpha;q)_\infty
\left(-\frac{q}{\alpha};q\right)_\infty
=
\frac{1}{(q;q)_\infty}
\sum_{m\in\mathbb Z}
q^{m(m-1)/2}\alpha^m.
\]

Since
\[
(-\alpha q;q)_\infty
=
\frac{(-\alpha;q)_\infty}{1+\alpha},
\]
the Jacobi triple product identity yields
\[
A(z,q)
=
P(q)\frac1{1+\alpha}
\sum_{m\in\mathbb Z}
q^{m(m-1)/2}\alpha^m .
\]

To extract the coefficient of $z^2$, we group together the terms
corresponding to $m$ and $1-m$. Since
\[
\frac{m(m-1)}2
=
\frac{(1-m)(-m)}2,
\]
we obtain
\[
A(z,q)
=
P(q)
\sum_{m\ge0}
F_m(z)\,
q^{m(m+1)/2},
\]
where
\[
F_m(z):=
\frac{\alpha^{m+1}+\alpha^{-m}}
{1+\alpha}.
\]

The next lemma determines the coefficient of $z^2$ in $F_m(z)$.

\begin{lemma}
	The polynomials $F_m(z)$ satisfy the recurrence
	\[
	F_{m+1}(z)=zF_m(z)-F_{m-1}(z)
	\qquad (m\ge1),
	\]
	with initial values
	\[
	F_0(z)=1,
	\qquad
	F_1(z)=z-1.
	\]
	
	Moreover, for every $k\ge1$,
	\[
	[z^2]F_m(z)
	=
	\begin{cases}
	(-1)^{k-1}\binom{k+1}{2}, & m=2k, \\[6pt]
	-(-1)^{k-1}\binom{k+1}{2}, & m=2k+1.
	\end{cases}
	\]
\end{lemma}

\begin{proof}
	Since $\alpha^2=z\alpha-1$, we have
	\[
	\alpha^{m+2}
	=
	z\alpha^{m+1}-\alpha^m,
	\]
	and similarly
	\[
	\alpha^{-(m+1)}
	=
	z\alpha^{-m}-\alpha^{-(m-1)}.
	\]
	Adding these identities and dividing by $1+\alpha$ yields
	\[
	F_{m+1}(z)=zF_m(z)-F_{m-1}(z).
	\]
	
	Writing
	\[
	F_m(z)=\sum_{r\ge0} c_m^{(r)} z^r,
	\]
	the recurrence implies
	\[
	c_{m+1}^{(2)}
	=
	c_m^{(1)}-c_{m-1}^{(2)},
	\qquad
	c_{m+1}^{(1)}
	=
	c_m^{(0)}-c_{m-1}^{(1)},
	\qquad
	c_{m+1}^{(0)}
	=
	-c_{m-1}^{(0)},
	\]
	with initial values determined by
	\[
	F_0(z)=1,
	\qquad
	F_1(z)=z-1.
	\]
	
	The first values are
	\[
	[z^2]F_m(z):
	\quad
	0,\;0,\;1,\;-1,\;-3,\;3,\;6,\;-6,\ldots
	\]
	for $m=0,1,\ldots,7$.
	These values suggest the stated formulas, which are readily verified by induction using the coefficient recurrences above.
\end{proof}

\medskip
Taking the coefficient of $z^2$ yields
\[
A(q)
=
P(q)
\sum_{k\ge0}
(-1)^k
\binom{k+2}{2}
q^{\binom{2k+3}{2}}
-
P(q)
\sum_{k\ge0}
(-1)^k
\binom{k+2}{2}
q^{\binom{2k+4}{2}}.
\]

Since
\[
\binom{k+2}{2}=T_{k+1},
\]
this becomes
\[
A(q)
=
P(q)
\left(
\sum_{k\ge0}
(-1)^k\,
T_{k+1}\,
q^{T_{2k+2}}
-
\sum_{k\ge0}
(-1)^k\,
T_{k+1}\,
q^{T_{2k+3}}
\right).
\]

For $r=2k+1$ we have
\[
(-1)^k=(-1)^{T_{r-1}},
\qquad
T_{k+1}=T_{\lceil r/2\rceil},
\]
while for $r=2k+2$ the same identities remain valid.

Hence the two sums combine into
\[
S(q)
=
\sum_{r\ge1}
(-1)^{T_{r-1}}\,
T_{\lceil r/2\rceil}\,
q^{T_{r+1}},
\]
which completes the proof of Theorem~\ref{thm:factorization}.

\section{Ramanujan-Type Congruence Modulo $3$}
\label{S4}

The congruence
\[
a(12n+11)\equiv 0 \pmod 3
\]
ultimately arises from the modular structure of the generating function
$A(q)$ modulo $3$.
The key observation is that the correction series $S(q)$ appearing in the
factorization
\[
A(q)=P(q)\,S(q)
\]
collapses modulo $3$ to a theta-type series. As a consequence,
$A(q)$ becomes congruent modulo $3$ to a function $H(q)$ that can be
expressed as the product of a Siegel function and an eta-quotient. This
identification allows us to embed the problem into the theory of modular
forms.

After converting $H(q)$ into a holomorphic modular form, we apply a
roots-of-unity filtration to isolate the coefficients in the arithmetic
progression $12n+11$. Sturm's theorem then reduces the congruence to a
finite computation.

\subsection {Reduction modulo $3$ to a convenient factorization}

We begin by reducing the factorization
\[
A(q)=P(q)\,S(q)
\]
modulo $3$. The crucial point is that the triangular-number coefficients
appearing in $S(q)$ simplify dramatically modulo $3$, causing the series
to collapse to a theta-type function.

A direct verification shows that
\[
T_{\lceil k/2 \rceil}
\equiv
\delta_{1,k \bmod 6}+\delta_{2,k \bmod 6}
\pmod{3},
\]
where $\delta_{i,j}$ denotes the Kronecker delta function. In other words, the triangular number $T_{\lceil k/2 \rceil}$ is $1$ modulo $3$ exactly when $k\equiv 1,2\pmod{6}$, and $0$ otherwise. Hence, modulo $3$, the series $S(q)$ reduces to
\[
S(q)\equiv
\sum_{k\ge0}(-1)^k
\Bigl(
q^{\binom{6k+3}{2}}
-
q^{\binom{6k+4}{2}}
\Bigr)
\pmod{3}.
\]
Rewriting this as a bilateral series (by reindexing the two residue classes) yields
\begin{align*}
S(q) \equiv \sum_{k\in\mathbb Z} (-1)^k\,q^{18k^2+15k+3} \pmod 3,
\end{align*}
Moreover, invoking the Jacobi triple product identity,  
\[
\sum_{k\in\mathbb Z} z^k q^{k^2}
=
(q^2;q^2)_\infty
(-zq;q^2)_\infty
(-q/z;q^2)_\infty,
\]
with $q$ replaced by $q^{18}$ and $z$ replaced  by $-q^{3}$,
we obtain
\begin{align*}
S(q) &\equiv q^3\,(q^3;q^{36})_\infty (q^{33};q^{36})_\infty (q^{36};q^{36})_\infty \pmod 3\\
&\equiv q^3\,(q^3;q^{36})_\infty (q^{33};q^{36})_\infty (q^{12};q^{12})^3_\infty \pmod 3.
\end{align*}
Hence 
\begin{align*}
A(q) \equiv  H(q) \pmod 3,
\end{align*}
where
\[
H(\tau)
=
q^{3}\,(q^3;q^{36})_\infty (q^{33};q^{36})_\infty 
\frac{(q^2;q^2)_\infty (q^{12};q^{12})^3_\infty}
{(q;q)_\infty (q^4;q^4)_\infty},
\qquad q=e^{2\pi i\tau}.
\]

For $a=(a_1,a_2)\in \mathbb{Q}^2 \setminus \mathbb{Z}^2$, $a\notin\mathbb{Z}^2$, the Siegel functions $g_a(\tau)$ are usually defined in terms of the Klein forms and the Dedekind eta functions (see Kubert--Lang \cite[Ch.~2, \S1]{KubertLang})  
\[
g_{a}(\tau)
=
-\,q^{B_2(a_1)/2}\, 
\zeta^{(a_1-1)/2}
\prod_{n=0}^\infty 
(1-q^{n+a_1}\zeta)
\prod_{n=1}^\infty (1-q^{n-a_1}\zeta^{-1}),
\]
where $\zeta=e^{2\pi i a_2}$ and $$B_2(x)=x^2-x+\frac16$$ is the second Bernoulli polynomial.
For a given integer $N$, we consider a class functions
\begin{align*}
E_r^{(N)}(\tau)=-g_{(r/N,0)}(N\tau) 
=q^{\frac{N}{2}B(r/N)}
\prod_{n=1}^\infty (1 - q^{(n-1)N + r})(1 - q^{nN - r}),
\end{align*}
for integers $r$ not congruent to $0$ modulo $N$, where $q=e^{2\pi i \tau}$. The fundamental properties of $E_r^{(N)}(\tau)$ were studied by Yang in \cite{Yang2004BLMS,yang2009modular}. In this context, we remark that
\begin{equation}
E_3(\tau)=E_3^{(36)}(\tau) = -g_{(1/12,0)}(36\tau)
=
\,q^{13/8}\,
(q^3;q^{36})_\infty
(q^{33};q^{36})_\infty.
\label{eq:siegel-identity}
\end{equation}

Recall Dedekind's eta function
\[
\eta(\tau)=q^{1/24}\prod_{n\ge1}(1-q^n),
\qquad
(q^m;q^m)_\infty=q^{-m/24}\,\eta(m\tau).
\]
A direct conversion yields
\[
\frac{(q^2;q^2)_\infty (q^{12};q^{12})^3_\infty}
{(q;q)_\infty (q^4;q^4)_\infty}
=
q^{-11/8}
\frac{\eta(2\tau)\,\eta(12\tau)^3}
{\eta(\tau)\, \eta(4\tau)}.
\]	
Combining with \eqref{eq:siegel-identity}, we obtain
\[
H(\tau)
=
-\,g_{(1/12,0)}(36\tau)\,
\frac{\eta(2\tau)\,\eta(12\tau)^3}
{\eta(\tau)\,\eta(4\tau)}.
\]

This representation reveals the modular structure underlying the
congruence. The function $H(\tau)$ is naturally expressed as the product
of a Siegel function and an eta-quotient. The first factor is a modular
unit on $\Gamma_1(36)$, while the second is an eta-quotient of weight
$1$. Consequently, one expects $H(\tau)$ itself to behave as a weight-$1$
modular form on $\Gamma_1(36)$.

The following lemmas establish these modularity properties.

\subsection{Modularity on $\Gamma_1(36)$} 

Considering that $H(\tau)$ is a product of a Siegel function and an $\eta$-quotient, we analyze the modular properties of the function $H(\tau)$ under the congruence subgroup $\Gamma_1(36)$. 

\begin{lemma}\label{lem:F-Gamma96}
	$E_3(\tau)$ is a weight-$0$ modular unit on $\Gamma_1(36)$ with some finite-order multiplier system $\nu_E$.
\end{lemma}

\begin{proof}
	Let
	\[
	\gamma=\begin{pmatrix}a&b\\ c&d\end{pmatrix}\in \Gamma_1(36).
	\]
	Then
	\[
	a\equiv d\equiv 1 \pmod{36},\qquad c\equiv 0 \pmod{36}.
	\]
	From \cite[Corollary 2]{Yang2004BLMS}, the function $E_3(\tau)$ satisfy the transformation law
	\[
	E_3(\gamma \tau)
	=
	\vartheta(a,36b,c/36,d)\,
	e^{\pi i\left(ab/4 - 3b\right)}
	E_{3a}(\tau),
	\]
	together with the relations
	\[
	E_{3+36} = E_3, \qquad E_{-3} = -E_3.
	\]
	Using the periodicity and symmetry relations above, it follows that
	\[
	E_{3a} = (-1)^{\lfloor a/36 \rfloor}E_3.
	\]
	Hence the transformation formula reduces to
	\[
	E_3(\gamma \tau)
	=
	\nu_E(\gamma)\, E_3(\tau),
	\]
	where
	\[
	\nu_E(\gamma)
	= (-1)^{\lfloor a/36 \rfloor}\, 
	\vartheta(a,36b,c/36,d)\,
	e^{\pi i\left(ab/4 - 3b\right)}
	\]
	with
	$$
	\vartheta(a,b,c,d) = -ie^{\pi i (ac(1-d^2)+d(b-c+3))/6 }.
	$$
	Thus $E_3$ transforms with weight $0$ on $\Gamma_1(36)$ with multiplier system $\nu_E$.
	
	Moreover, the factor $\nu_E(\gamma)$ is a root of unity, since both $\vartheta(a,36b,c/36,d)$ and the exponential term are roots of unity. Hence $\nu_E$ has finite order.
	
	Finally, Siegel functions are holomorphic and nonvanishing on~$\mathfrak H$, and their zeros and poles occur only at cusps. Therefore the same is true for $E_3(\tau)$. It follows that $E_3$ is a weight-$0$ modular unit on $\Gamma_1(36)$ with finite-order multiplier system~$\nu_E$.
\end{proof}

\begin{lemma}\label{lem:G-Gamma36}
	Define
	\[
	G(\tau):=\frac{\eta(2\tau)\,\eta(12\tau)^3}{\eta(\tau)\,\eta(4\tau)}.
	\]
	Then $G(\tau)$ is a weakly holomorphic modular form of weight $1$ on $\Gamma_1(36)$
	with some finite-order multiplier system $\nu_G$.
\end{lemma}

\begin{proof}
	Let
	\[
	\gamma=\begin{pmatrix}a&b\\ c&d\end{pmatrix}\in \Gamma_1(36).
	\]
	Then
	\[
	a\equiv d\equiv 1 \pmod{36},\qquad c\equiv 0 \pmod{36}.
	\]
	For each $k\in\{1,2,4,12\}$, set
	\[
	\gamma_k:=
	\begin{pmatrix}
	a & kb\\
	c/k & d
	\end{pmatrix}.
	\]
	Since $k\mid 36$ and $c\equiv 0\pmod{36}$, we have $c/k\in \mathbb Z$, and
	\[
	\det(\gamma_k)=ad-bc=1,
	\]
	so $\gamma_k\in SL_2(\mathbb Z)$.
	
	Now let $z_k:=k\tau$. Then
	\[
	k(\gamma\tau)
	=\frac{k(a\tau+b)}{c\tau+d}
	=\frac{az_k+kb}{(c/k)z_k+d}
	=\gamma_k(z_k).
	\]
	Hence, by the transformation law for Dedekind's eta-function,
	\[
	\eta(k\gamma\tau)
	=
	\eta(\gamma_k z_k)
	=
	\varepsilon(\gamma_k)\,\bigl((c/k)z_k+d\bigr)^{1/2}\,\eta(z_k),
	\]
	where $\varepsilon(\gamma_k)$ is a $24$th root of unity. Since
	\[
	(c/k)z_k+d=(c/k)(k\tau)+d=c\tau+d,
	\]
	this becomes
	\[
	\eta(k\gamma\tau)
	=
	\varepsilon(\gamma_k)\,(c\tau+d)^{1/2}\,\eta(k\tau).
	\]
	
	Applying this for each $k\in\{1,2,4,12\}$, we obtain
	\[
	G(\gamma\tau)
	=
	\frac{\eta(2\gamma\tau)\,\eta(12\gamma\tau)^3}
	{\eta(\gamma\tau)\,\eta(4\gamma\tau)}
	=
	\frac{\varepsilon(\gamma_2)\,\varepsilon(\gamma_{12})^3}
	{\varepsilon(\gamma_1)\,\varepsilon(\gamma_4)}
	(c\tau+d)^{(1+3-1-1)/2}\,
	\frac{\eta(2\tau)\,\eta(12\tau)^3}{\eta(\tau)\,\eta(4\tau)}.
	\]
	Thus
	\[
	G(\gamma\tau)=\nu_G(\gamma)\,(c\tau+d)\,G(\tau),
	\]
	where
	\[
	\nu_G(\gamma):=
	\frac{\varepsilon(\gamma_2)\,\varepsilon(\gamma_{12})^3}
	{\varepsilon(\gamma_1)\,\varepsilon(\gamma_4)}
	\]
	is a root of unity. Therefore $G$ is a weakly holomorphic modular form of
	weight $1$ on $\Gamma_1(36)$ with finite-order multiplier system $\nu_G$.
	
	Finally, the Dedekind eta-function is holomorphic and nonvanishing on
	$\mathfrak H$. Hence the quotient
	\[
	G(\tau)=\frac{\eta(2\tau)\,\eta(12\tau)^3}{\eta(\tau)\,\eta(4\tau)}
	\]
	is also holomorphic and nonvanishing on $\mathfrak H$. It follows that any
	zeros or poles of $G$ can occur only at the cusps. Therefore
	$G\in M_1^{!}(\Gamma_1(36),\nu_G)$.
\end{proof}

\begin{lemma}\label{lem:H-Gamma36}
	The function $H(\tau)$
	is a weakly holomorphic modular form of weight $1$ on $\Gamma_1(36)$
	with some finite-order multiplier system $\nu_H$.
\end{lemma}

\begin{proof}
	Combining the two transformations,
	\[
	H(\gamma\tau)
	=
	E_3(\gamma\tau)\,G(\gamma\tau)
	=
	\nu_E(\gamma)\,\nu_G(\gamma)\,(c\tau+d)\,H(\tau),
	\]
	so $H\in M_1^{!}(\Gamma_1(36),\nu_H)$ with $\nu_H:=\nu_E\nu_G$.
	Since both $E_3$ and $G$ are holomorphic and nonvanishing on $\mathfrak H$,
	the same holds for $H$, and any zeros/poles occur only at cusps.
\end{proof}

\begin{lemma}\label{lem:Phi-holomorphic}
	Let
	\[
	\Phi(\tau):=H(\tau)\,\eta(12\tau)^4.
	\]
	Then $\Phi(\tau)$ is a holomorphic modular form of weight $3$ on
	$\Gamma_1(36)$ with some finite-order multiplier system $\nu_\Phi$.
\end{lemma}

\begin{proof}
	By Lemma~\ref{lem:H-Gamma36}, the function \(H(\tau)\) is a weakly
	holomorphic modular form of weight \(1\) on \(\Gamma_1(36)\) with some
	finite-order multiplier system \(\nu_H\).  Moreover, by the transformation
	law of Dedekind's eta-function, the function \(\eta(12\tau)^4\) is a modular
	form of weight \(2\) on \(\Gamma_1(36)\) with some finite-order multiplier
	system \(\nu_\eta\).  Therefore their product
	\[
	\Phi(\tau)=H(\tau)\,\eta(12\tau)^4
	\]
	satisfies
	\[
	\Phi(\gamma\tau)
	=
	\nu_H(\gamma)\,\nu_\eta(\gamma)\,(c\tau+d)^3\,\Phi(\tau)
	\qquad
	\left(
	\gamma=\begin{pmatrix}a&b\\ c&d\end{pmatrix}\in\Gamma_1(36)
	\right).
	\]
	Hence \(\Phi\) is a weakly holomorphic modular form of weight \(3\) on
	\(\Gamma_1(36)\) with finite-order multiplier system
	\[
	\nu_\Phi:=\nu_H\,\nu_\eta.
	\]
	Since \(H(\tau)\) is holomorphic on \(\mathfrak H\) and \(\eta(12\tau)\) is
	holomorphic and nonvanishing on \(\mathfrak H\), it follows immediately that
	\(\Phi(\tau)\) is holomorphic on the upper half-plane.  Thus the only point
	that remains is to verify holomorphy at the cusps.

	For convenience, set
	\[
	\Phi(\tau)=E_3(\tau)\,F(\tau),
	\]
	where
	\[
	F(\tau)=\frac{\eta(2\tau)\,\eta(12\tau)^7}{\eta(\tau)\,\eta(4\tau)}.
	\]
	We shall compute the order at each cusp separately for \(E_3\) and \(F\), and
	then add the results.  Since \(\nu_\Phi\) has finite order, the orders at the
	cusps may be rational numbers; nonnegativity of all cusp orders is exactly the
	condition needed for holomorphy.
	
	At the cusp \(\infty\), the order is read directly from the leading
	\(q\)-power.  
	By the product formula for the Siegel function,
	\[
	E_3(\tau)=-g_{(1/12,0)}(36\tau)
	=
	\,q^{13/8}\,(q^3;q^{36})_\infty\,(q^{33};q^{36})_\infty.
	\]
	we obtain
	\[
	\ord_\infty(E_3)=\frac{13}{8}.
	\tag{3}\label{eq:F-infty}
	\]
	Since
	\begin{align*}
	F(\tau)
	& =
	q^{\,1/{12}+7/2-1/{24}-1/6}\,
	\frac{(q^2;q^2)_\infty\,(q^{12};q^{12})_\infty^7}
	{(q;q)_\infty\,(q^4;q^4)_\infty}\\
	& =
	q^{27/8}\cdot(\text{a power series in }q\text{ with constant term }1),
	\end{align*}
	we obtain
	\[
	\ord_\infty(F)=\frac{27}{8}.
	\tag{2}\label{eq:E-infty}
	\]
	At \(\infty\), combining \eqref{eq:E-infty} and \eqref{eq:F-infty}, we find
	\[
	\ord_\infty(\Phi)
	=
	\ord_\infty(E_3)+\ord_\infty(F)
	=
	\frac{13}{8}+\frac{27}{8}
	=
	5.
	\]
	
	For a finite cusp \(a/c\) with \(c\mid 36\) and \(\gcd(a,c)=1\), the standard
	cusp-order formula for this Siegel function gives (see Yang \cite[Proposition 3]{yang2009modular})
	\begin{align}
	\ord_{a/c}(E_3)
	=
	\frac{\gcd(c,36)}{2}\,
	B_2\!\left(\Big\{\frac{3a}{\gcd(c,36)}\Big\}\right)
	= 	\frac{c}{2}\,
	B_2\!\left(\Big\{\frac{3a}{c}\Big\}\right).
	\label{eq:F-cusp-order} 
	\end{align}
	On the other hand, 
	the function \(F(\tau)\) is an eta-quotient of level \(36\), namely
	\[
	F(\tau)=\prod_{\delta\mid 36}\eta(\delta\tau)^{r_\delta},
	\]
	where the only nonzero exponents are
	\[
	r_1=-1,\qquad r_2=1,\qquad r_4=-1,\qquad r_{12}=7.
	\]
	For an eta-quotient on \(\Gamma_1(36)\), the standard cusp-order formula \cite{Raji}
	states that if \(a/c\) is a cusp with \(c\mid 36\) and \(\gcd(a,c)=1\), then
	\begin{align}
	\ord_{a/c}(F)
	& =
	\frac{36}{24\,\gcd(c,36)}
	\sum_{\delta\mid 36}
	r_\delta \frac{\gcd(c,\delta)^2}{\delta}\notag \\
	& =
	\frac{3}{2c}
	\sum_{\delta\mid 36}
	r_\delta \frac{\gcd(c,\delta)^2}{\delta}
	.\label{eq:E-cusp-order}
	\end{align}
	A complete set of representatives for the cusps of \(\Gamma_1(36)\) is
	\[
	\infty,\ \frac12,\ \frac13,\ \frac23,\ \frac14,\ \frac16,\ \frac56,\ \frac19,\ \frac1{12},\ \frac5{12},\ \frac1{18},\ \frac1{36}.
	\]
	Indeed, for \(\Gamma_1(N)\), cusps with \(\gcd(c,N)=d\) are parametrized by classes of \(a\) modulo \(\gcd(d,N/d)\), and for \(N=36\) this yields exactly
	\[
	\sum_{d\mid 36}\varphi(\gcd(d,36/d))=12
	\]
	inequivalent cusps (see e.g. \cite[Prop.~3.8.3]{DiamondShurman}).
	Evaluating \eqref{eq:E-cusp-order} and \eqref{eq:F-cusp-order} at the cusps of $\Gamma_1(36)$, we obtain the values listed in the following table:
	\[
	\begin{array}{c|cccccccccccc}
	s 
	& \infty & \frac{1}{2} & \frac{1}{3} & \frac{2}{3} & \frac{1}{4} & \frac{1}{6} & \frac{5}{6} & \frac{1}{9} & \frac{1}{12} & \frac{5}{12} & \frac{1}{18} & \frac{1}{36} \\ 
	\noalign{\vskip 4pt}
	\hline 
	\noalign{\vskip 4pt}
	\ord_s(E_3) 
	& \frac{13}{8} & -\frac{1}{12} & \frac{1}{4} & \frac{1}{4} & -\frac{1}{24} & -\frac{1}{4} & -\frac{1}{4} & -\frac{1}{4} & -\frac{1}{8} & -\frac{1}{8} & \frac{1}{4} & \frac{13}{8} \\[6pt]
	
	\ord_s(F) 
	& \frac{27}{8} & \frac{7}{4} & \frac{9}{4} & \frac{9}{4} & \frac{19}{8} & \frac{21}{4} & \frac{21}{4} & \frac{3}{4} & \frac{81}{8} & \frac{81}{8} & \frac{7}{4} & \frac{27}{8} \\[6pt]
	
	\ord_s(\Phi) 
	& 5 & \frac{5}{3} & \frac{5}{2} & \frac{5}{2} & \frac{7}{3} & 5 & 5 & \frac{1}{2} & 10 & 10 & 2 & 5
	\end{array}
	\]
	
	We have shown that \(\Phi(\tau)\) transforms as a modular form of weight \(3\)
	on \(\Gamma_1(36)\) with finite-order multiplier system \(\nu_\Phi\), is
	holomorphic on \(\mathfrak H\), and is holomorphic at every cusp (every entry in the last row of table is nonnegative).  Hence
	$
	\Phi(\tau)\in M_3(\Gamma_1(36),\nu_\Phi).
	$
	This proves the lemma.
\end{proof}

\subsection{Passing to a holomorphic modular form}

In this subsection we use a standard \emph{roots-of-unity filter} to extract from
the $q$-series of $H(\tau)$ only those coefficients whose exponents lie in a
prescribed residue class modulo $12$. We then transfer this filtration to the
holomorphic modular form $\Phi(\tau):=H(\tau)\,\eta(12\tau)^4$, which allows us
to conclude a congruence by working entirely within the holomorphic category.

Let
\[
H(\tau)=\sum_{n\ge 0} h(n)\,q^n
\]
denote the $q$-expansion of $H$. We define the function $H_{11}(\tau)$ by
\[
H_{11}(\tau)
:=
\frac{1}{12}\sum_{r=0}^{11}\zeta_{12}^{-11r}\,
H\!\left(\tau+\frac{r}{12}\right),
\qquad
\zeta_{12}=e^{2\pi i/12}.
\]

Since
\[
H(\tau)=\sum_{n\ge0}h(n)\,q^n,
\]
the definition of \(H_{11}\) immediately gives
\[
H_{11}(\tau)=\sum_{n\ge0} h(12n+11)\,q^{12n+11}.
\]
Therefore, it suffices to prove that
\[
H_{11}(\tau)\equiv 0 \pmod 3.
\]

Define similarly
\[
\Phi_1(\tau)
:=
\frac1{12}\sum_{r=0}^{11}\zeta_{12}^{-r}\,
\Phi\!\left(\tau+\frac r{12}\right).
\]
Since
\[
\Phi(\tau)=H(\tau)\,\eta(12\tau)^4,
\]
and Dedekind's eta-function satisfies
\[
\eta(\tau+1)=e^{\pi i/12}\,\eta(\tau),
\]
we have
\[
\eta\!\left(12\Bigl(\tau+\frac r{12}\Bigr)\right)^4
=
\eta(12\tau+r)^4
=
e^{\pi i r/3}\,\eta(12\tau)^4
=
\zeta_{12}^{2r}\,\eta(12\tau)^4.
\]
Hence
\begin{align*}
\Phi_1(\tau)
&=
\frac1{12}\sum_{r=0}^{11}\zeta_{12}^{-r}
H\!\left(\tau+\frac r{12}\right)
\eta\!\left(12\Bigl(\tau+\frac r{12}\Bigr)\right)^4 \\
&=
\eta(12\tau)^4\,
\frac1{12}\sum_{r=0}^{11}\zeta_{12}^{(2-1)r}
H\!\left(\tau+\frac r{12}\right) \\
&=
\eta(12\tau)^4\,
\frac1{12}\sum_{r=0}^{11}\zeta_{12}^{r}\,
H\!\left(\tau+\frac r{12}\right).
\end{align*}
But
$
\zeta_{12}^{r}=\zeta_{12}^{-11r}
$,
so by the definition of \(H_{11}\),
\[
\Phi_1(\tau)=\eta(12\tau)^4\,H_{11}(\tau).
\]
In particular,
\[
\Phi_1(\tau)=\sum_{n\ge0} c(12n+1)\,q^{12n+1},
\]
where \(c(m)\) denotes the coefficient of \(q^m\) in \(\Phi(\tau)\).

Now
\[
\eta(12\tau)^4=q^2\prod_{n\ge1}(1-q^{12n})^4
\]
has the form
\[
q^2\cdot U(q^{12}),
\qquad U(X)\in 1+X\mathbb Z[[X]],
\]
so \(U(X)\) is invertible in \(\mathbb Z[[X]]\). Therefore
\[
\Phi_1(\tau)\equiv 0 \pmod 3
\qquad\Longleftrightarrow\qquad
H_{11}(\tau)\equiv 0 \pmod 3.
\]
Thus it remains to show that \(\Phi_1\equiv0\pmod3\).

\subsection{Modularity of the translated forms}

In order to apply Sturm's theorem to the filtered form
\[
\Phi_1(\tau)=\frac1{12}\sum_{r=0}^{11}\zeta_{12}^{-r}\,
\Phi\!\left(\tau+\frac r{12}\right),
\]
we must verify that the translated functions
\(\Phi(\tau+j/12)\) remain modular on a common congruence subgroup.
The following lemma shows that all such translations are modular
forms of weight \(3\) on a fixed subgroup of \(SL_2(\mathbb Z)\).

For \(j\in\{0,1,\dots,11\}\), define
\[
\Psi_j(\tau):=\Phi\!\left(\tau+\frac{j}{12}\right).
\]
We show that each translated form \(\Psi_j\) is modular on a common congruence subgroup of \(SL_2(\mathbb Z)\).

\begin{lemma}\label{lem:translated-forms}
	For each $j\in\{0,1,\dots,11\}$, the function \(\Psi_j\) is a modular form of weight \(3\) on
	$\Gamma_0(432)\cap\Gamma_1(36)$
	with some finite-order multiplier system.
\end{lemma}

\begin{proof}
	Fix \(j\in\{0,1,\dots,11\}\), and write
	\[
	\alpha=\frac{j}{12},
	\qquad
	T_\alpha=
	\begin{pmatrix}
	1 & \alpha\\
	0 & 1
	\end{pmatrix}.
	\]
	Then
	\[
	\Psi_j(\tau)=\Phi(T_\alpha\tau).
	\]
	
	Let
	\[
	\gamma=
	\begin{pmatrix}
	a & b\\
	c & d
	\end{pmatrix}
	\in \Gamma_0(432)\cap\Gamma_1(36).
	\]
	By definition, this means
	\[
	c\equiv 0 \pmod{432},
	\qquad
	a\equiv d\equiv 1 \pmod{36}.
	\]
	It is straightforward to verify that
	\[
	T_\alpha \gamma T_{-\alpha}
	=
	\begin{pmatrix}
	a+\alpha c &
	b+\alpha(d-a)-\alpha^2 c\\
	c &
	d-\alpha c
	\end{pmatrix}
	\in \Gamma_1(36).
	\]
	
	
	
	Since \(\Phi\in M_3(\Gamma_1(36),\nu_\Phi)\), it follows that
	\begin{align*}
	\Psi_j(\gamma\tau)
	&=
	\Phi(T_\alpha\gamma\tau) \\
	&=
	\Phi\bigl((T_\alpha\gamma T_{-\alpha})(T_\alpha\tau)\bigr) \\
	&=
	\nu_\Phi(T_\alpha\gamma T_{-\alpha})\,
	\bigl(c(T_\alpha\tau)+d-\alpha c\bigr)^3
	\Phi(T_\alpha\tau).
	\end{align*}
	Now
	\[
	c(T_\alpha\tau)+d-\alpha c
	=
	c\!\left(\tau+\alpha\right)+d-\alpha c
	=
	c\tau+d,
	\]
	so we obtain
	\[
	\Psi_j(\gamma\tau)
	=
	\nu_j(\gamma)\,(c\tau+d)^3\,\Psi_j(\tau),
	\]
	where
	\[
	\nu_j(\gamma):=\nu_\Phi(T_\alpha\gamma T_{-\alpha})
	\]
	is a finite-order multiplier system. This proves the lemma.
\end{proof}

By the argument of Lemma \ref{lem:translated-forms}, for any
\[
\gamma=\begin{pmatrix}a&b\\c&d\end{pmatrix}\in \Gamma_0(432)\cap\Gamma_1(36),
\]
and $\alpha=j/12$ with $j=0,1,\dots,11$, one has
\[
\Psi_j(\gamma\tau)
=
\nu_\Phi(T_\alpha\gamma T_{-\alpha})\,(c\tau+d)^3\,\Psi_j(\tau),
\]
since $T_\alpha\gamma T_{-\alpha}\in\Gamma_1(36)$.
The multiplier on the right-hand side depends \emph{a priori} on $j$. 
We now show that there exists a congruence subgroup on which this dependence disappears.

\begin{lemma}\label{Lema4.6}
	Let $\gamma \in \Gamma_0(3456)\cap \Gamma_1(288)$.   For $\alpha = j/12$ with $j \in \{0,1, \dots, 11\}$, the multiplier of $\Phi$ satisfies
	\[
	\nu_\Phi(T_\alpha \gamma T_{-\alpha}) = \nu_\Phi(\gamma).
	\]
	In particular,
	\[
	\Psi_j(\gamma\tau)=\nu_\Phi(\gamma)\,(c\tau+d)^3\,\Psi_j(\tau).
	\]
\end{lemma}

\begin{proof}
	Let
	\[
	\gamma=
	\begin{pmatrix}
	a & b\\
	c & d
	\end{pmatrix}
	\qquad\text{and}\qquad
	\gamma' := T_\alpha \gamma T_{-\alpha}.
	\]
	A direct computation gives
	\[
	\gamma'=
	\begin{pmatrix}
	a+\alpha c & b+\alpha(d-a)-\alpha^2c\\
	c & d-\alpha c
	\end{pmatrix}.
	\]
	Since $\Phi = -E_3 F$, it suffices to show that the multipliers of $E_3$ and $F$ are separately invariant under the replacement $\gamma \mapsto \gamma'$, i.e.,
	\[
	\nu_E(\gamma')=\nu_E(\gamma)
	\quad\text{and}\quad
	\nu_F(\gamma')=\nu_F(\gamma).
	\]
	
	\medskip
	
	\noindent
	\textit{The $\eta$-quotient $F$.}
	From Lemma \ref{lem:G-Gamma36}, the multiplier of $F$ is
	\[
	\nu_F(\gamma)=
	\frac{\varepsilon(\gamma_2)\,\varepsilon(\gamma_{12})^7}
	{\varepsilon(\gamma_1)\,\varepsilon(\gamma_4)},
	\qquad
	\gamma_\delta=
	\begin{pmatrix}
	a & \delta b\\
	c/\delta & d
	\end{pmatrix},
	\]
	and similarly for $\gamma'$.
	
	For $\gamma'$ we obtain
	\[
	\gamma'_\delta=
	\begin{pmatrix}
	a+\alpha c &
	\delta b+\delta\alpha(d-a)-\delta\alpha^2c\\
	c/\delta &
	d-\alpha c
	\end{pmatrix}.
	\]
	Hence
	\[
	\gamma'_\delta - \gamma_\delta =
	\begin{pmatrix}
	\alpha c & \delta\alpha(d-a)-\delta\alpha^2c\\
	0 & -\alpha c
	\end{pmatrix}.
	\]
	
	Since $\gamma \in \Gamma_0(3456)\cap \Gamma_1(288)$, we have
	\[
	c \equiv 0 \pmod{3456}, \qquad a \equiv d \equiv 1 \pmod{288}.
	\]
	Using $\alpha = j/12$, it follows that
	\[
	\alpha c \equiv 0 \pmod{24},\qquad
	\alpha(d-a)\equiv 0 \pmod{24},\qquad
	\alpha^2 c \equiv 0 \pmod{24}.
	\]
	Therefore, for each $\delta \in \{1,2,4,12\}$,
	\[
	\gamma'_\delta \equiv \gamma_\delta \pmod{24}.
	\]
	The Dedekind eta-function satisfies
	\[
	\eta(\gamma\tau)=\varepsilon(\gamma)\,(c\tau+d)^{1/2}\,\eta(\tau),
	\]
	where the multiplier $\varepsilon(\gamma)$ is defined as 
	\begin{align*}
	\varepsilon(\gamma):=
	\begin{cases}
	e^{\frac{bi\pi}{12}}, & \text{for $c=0$, $d=1$,}\\
	e^{i\pi \left(\frac{a+d}{12c}-s(d,c)-\frac14\right)}, & \text{for $c>0$,}
	\end{cases}
	\end{align*}
	with
	\[
	s(d,c)=\sum_{n=1}^{c-1}
	\Bigl(\!\Bigl(\frac{n}{c}\Bigr)\!\Bigr)
	\Bigl(\!\Bigl(\frac{dn}{c}\Bigr)\!\Bigr)
	\]
	and
	\[
	((x))=
	\begin{cases}
	x-\lfloor x\rfloor-\frac12,& x\notin \mathbb Z,\\[4pt]
	0,& x\in\mathbb Z.
	\end{cases}
	\]
	For each $\delta \in \{1,2,4,12\}$, a direct computation shows that
	\[
	\varepsilon(\gamma'_\delta)=\varepsilon(\gamma_\delta).
	\]
	Hence
	\[
	\nu_F(\gamma')=\nu_F(\gamma).
	\]
	\medskip
	
	\noindent
	\textit{The Siegel function $E_3$.}
	Let
	\[
	\gamma_0=
	\begin{pmatrix}
	a & 36b\\
	c/36 & d
	\end{pmatrix},
	\qquad
	\gamma_0'=
	\begin{pmatrix}
	a+\alpha c & 36b+36\alpha(d-a)-36\alpha^2c\\
	c/36 & d-\alpha c
	\end{pmatrix}.
	\]
	First, we compare $\gamma_0$ and $\gamma_0'$ modulo $12$:
	\[
	\gamma_0' - \gamma_0 =
	\begin{pmatrix}
	\alpha c & 36\alpha(d-a)-36\alpha^2c\\
	0 & -\alpha c
	\end{pmatrix}.
	\]
	From the congruences above, we have $\alpha c \equiv 0 \pmod{12}$, and the upper-right entry is clearly divisible by $12$. Hence
	\[
	\gamma_0' \equiv \gamma_0 \pmod{12}.
	\]
	From Lemma \ref{lem:F-Gamma96}, the multiplier of $E_3$ is of the form
	\[
	\nu_E(\gamma)
	= (-1)^{\lfloor a/36 \rfloor}\, 
	\vartheta(a,36b,c/36,d)\,
	e^{\pi i\left(ab/4 - 3b\right)}
	\]
	with
	$$
	\vartheta(a,b,c,d) = -ie^{\pi i (ac(1-d^2)+d(b-c+3))/6 },
	$$
	and similarly for $\gamma'$.
	It follows that
	\[
	\nu_E(\gamma')=\nu_E(\gamma).
	\]
	
	Combining the two parts, we obtain
	\[
	\nu_\Phi(\gamma')
	=
	\nu_F(\gamma')\,\nu_E(\gamma')
	=
	\nu_F(\gamma)\,\nu_E(\gamma)
	=
	\nu_\Phi(\gamma),
	\]
	which proves the lemma.
\end{proof}

It follows that all translated forms $\Psi_j$ transform with the \emph{same} multiplier $\nu_\Phi$ on $\Gamma_0(3456)\cap\Gamma_1(288)$. In particular,
\[
\Psi_j\in M_3(\Gamma_0(3456)\cap\Gamma_1(288),\nu_\Phi)
\qquad (j\in\{0,1,\dots,11\}).
\]

Since the filtered form
\[
\Phi_1(\tau):=
\frac1{12}\sum_{r=0}^{11}\zeta_{12}^{-r}\,
\Phi\!\left(\tau+\frac r{12}\right)
\]
is a finite linear combination of the translated forms \(\Psi_r\), Lemma~\ref{Lema4.6} implies that
\[
\Phi_1\in M_3(\Gamma_0(3456)\cap\Gamma_1(288),\nu_\Phi)
\]
for some finite-order multiplier system \(\nu_\Phi\).

\subsection{Application of Sturm's theorem.}

Having shown that the filtered form
$
\Phi_1(\tau)
$
is a holomorphic modular form of weight \(3\) on the congruence subgroup
\(\Gamma:=\Gamma_0(3456)\cap\Gamma_1(288)\),
we can apply Sturm's theorem to deduce the desired congruence.
To determine the Sturm bound we first compute the index of this subgroup
in \(SL_2(\mathbb Z)\).

Since \(\Gamma\) is a congruence subgroup of level \(3456\), its index can be computed from the reduction modulo \(3456\). By the Chinese remainder theorem,
\[
SL_2(\mathbb Z/3456\mathbb Z)
\cong
SL_2(\mathbb Z/128\mathbb Z)\times SL_2(\mathbb Z/27\mathbb Z),
\]
because \(3456=2^7\cdot3^3\).

For a prime power modulus one has (see e.g. \cite[Ch.~2]{DiamondShurman})
\[
|SL_2(\mathbb Z/p^n\mathbb Z)|=p^{3n-2}\,(p^2-1).
\]
Hence
\[
|SL_2(\mathbb Z/128\mathbb Z)|=2^{19}(4-1)=2^{19}\cdot 3,
\]
and
\[
|SL_2(\mathbb Z/27\mathbb Z)|=3^{7}(9-1)=2^3\cdot3^7.
\]
Therefore
\[
|SL_2(\mathbb Z/3456\mathbb Z)|
=
2^{22}\cdot 3^8
\]

We now determine the image of \(\Gamma\) modulo \(3456\).
The conditions defining \(\Gamma\) are
\[
c\equiv0\pmod{3456},
\qquad
a\equiv d\equiv1\pmod{288}.
\]
Thus every element of the reduction of \(\Gamma\) modulo \(3456\) has the form
\[
\begin{pmatrix}
a & b\\
0 & a^{-1}
\end{pmatrix}
\pmod{3456},
\]
where \(a\) is a unit modulo \(3456\) satisfying \(a\equiv1\pmod{288}\).
Since
\[
\varphi(3456)=1152
\quad\text{and}\quad
\varphi(288)=96,
\]
the number of such units is
\[
\frac{\varphi(3456)}{\varphi(288)}=12.
\]
The entry \(b\) is arbitrary modulo \(3456\), so there are \(3456\) choices.

Hence
\[
|\Gamma \bmod 3456|=12\cdot3456=2^9\cdot 3^4.
\]
Consequently
\[
[SL_2(\mathbb Z):\Gamma]
=
\frac{|SL_2(\mathbb Z/3456\mathbb Z)|}{|\Gamma \bmod 3456|}
=
\frac{2^{22}\cdot 3^8}{2^9\cdot 3^4}
= 2^{13}\cdot 3^4
=
663552.
\]

Since \(\Phi_1\) is a holomorphic modular form of weight \(3\) on \(\Gamma\),
Sturm's theorem \cite{Sturm} implies that it suffices to verify the congruence
\[
\Phi_1(\tau)\equiv0\pmod3
\]
for the coefficients up to the Sturm bound
\[
B=\frac{3}{12}\cdot2^{13}\cdot3^4=2^{11}\cdot 3^4=165888.
\]

Now
\[
\Phi_1(\tau)=\sum_{n\ge0}c(12n+1)\,q^{12n+1},
\]
so it is enough to verify that
\[
c(12n+1)\equiv0\pmod3
\qquad
(12n+1\le165888).
\]
Equivalently,
\[
c(12n+1)\equiv0\pmod3
\qquad
(0\le n\le13823).
\]

Using SageMath, we verified that
\[
c(12n+1)\equiv0\pmod3
\qquad (0\le n\le13823).
\]
Therefore, by Sturm's theorem \cite{Sturm},
\[
\Phi_1(\tau)\equiv0\pmod3.
\]

Since
\[
\Phi_1(\tau)=\eta(12\tau)^4\,H_{11}(\tau)
\]
and $\eta(12\tau)^4=q^2U(q^{12})$ with
\(U(X)\in1+X\mathbb Z[[X]]\) invertible,
we obtain
\[
H_{11}(\tau)\equiv0\pmod3.
\]
Hence
\[
h(12n+11)\equiv0\pmod3
\qquad(n\ge0).
\]

Finally, since \(A(q)\equiv H(q)\pmod3\), it follows that
\[
a(12n+11)\equiv0\pmod3
\qquad(n\ge0).
\]
This completes the proof of Theorem~\ref{Th3}.

\section{Concluding Remarks}

In this paper we investigated the arithmetic and combinatorial properties of the two-block odd partition function $a(n)$. Despite its definition as a signed enumeration, we established that $a(n)$ is nonnegative for all $n$, revealing a striking positivity phenomenon within a highly constrained partition model.

\smallskip

We also derived a structural factorization of the generating function
\[
A(q)=P(q)\,S(q),
\]
where $P(q)$ is the generating function for partitions with distinct odd parts and unrestricted even parts, and $S(q)$ is a theta-type correction series involving triangular numbers. This decomposition highlights an unexpected interaction between restricted partition functions and theta-type series. Furthermore, we proved the congruence
\[
a(12n+11)\equiv0\pmod3,
\]
revealing an additional layer of arithmetic regularity.

\smallskip

These results suggest that partition functions defined by simultaneous restrictions on support and multiplicities may exhibit richer structure than previously anticipated. It would be of interest to investigate analogous constructions involving more than two distinct part sizes and to determine whether similar positivity phenomena persist in broader families of signed partition functions.

\smallskip

The congruence modulo $3$ appears to be structural rather than accidental. Indeed, its proof ultimately rests on the factorization
\(
A(q)=P(q)\,S(q),
\)
together with the observation that the triangular-number correction series $S(q)$ collapses modulo $3$ to a theta-type series. This reduction makes it possible to interpret the generating function modulo $3$ in terms of modular objects, leading naturally to the residue class $12n+11$ through a roots-of-unity filtration.

\smallskip

It would be interesting to determine whether analogous reductions of $S(q)$ modulo other primes give rise to further Ramanujan-type congruences. More generally, the factorization $A(q)=P(q)\,S(q)$ suggests that the arithmetic of the two-block odd partition function may be governed by a broader family of modular phenomena. In particular, one may ask whether similar congruences occur for partition functions involving more than two distinct part sizes or different parity restrictions on multiplicities.

\section*{Declarations}
\begin{itemize}
	\item \textbf{Conflict of interest}: The author declares no conflict of interest.
	\item \textbf{Ethics approval and consent to participate}: Not applicable.
	\item \textbf{Consent for publication}: Not applicable.
	\item \textbf{Data availability}: Not applicable.
	\item \textbf{Materials availability}: Not applicable.
	\item \textbf{Code availability}: Not applicable.
	\item \textbf{Author contribution}: The author contributed to all aspects of the manuscript.
	\item \textbf{Funding}: No funding was received for this work.
\end{itemize}


\begin{thebibliography}{99}

\bibitem{AAB}
T.~Amdeberhan, G.~E. Andrews and C.~Ballantine,
Lambert series and double Lambert series,
\emph{J. Combin. Theory Ser. A} \textbf{221} (2026), 106154.

\bibitem{Sloane}
N.~J.~A. Sloane,
The On-Line Encyclopedia of Integer Sequences,
available at \url{https://oeis.org}.

\bibitem{Andrews76}
G.~E. Andrews,
\emph{The Theory of Partitions},
Encyclopedia of Mathematics and its Applications, Vol.~2,
Addison--Wesley, Reading, MA, 1976.

\bibitem{HardyWright}
G.~H. Hardy and E.~M. Wright,
\emph{An Introduction to the Theory of Numbers},
6th ed.,
Oxford University Press, Oxford, 2008.

\bibitem{Johnson}
W.~P. Johnson,
\emph{An Introduction to $q$-Analysis},
American Mathematical Society, Providence, RI, 2020.

\bibitem{KubertLang}
D.~S. Kubert and S.~Lang,
\emph{Modular Units},
Grundlehren der mathematischen Wissenschaften, Vol.~244,
Springer, New York, 1981.

\bibitem{Yang2004BLMS}
Y.-F. Yang,
Transformation formulas for generalized Dedekind eta functions,
\emph{Bull. Lond. Math. Soc.} \textbf{36} (2004), no.~5, 671--682.

\bibitem{yang2009modular}
Y.-F. Yang,
Modular units and cuspidal divisor class groups of $X_1(N)$,
\emph{J. Algebra} \textbf{322} (2009), no.~2, 514--553.

\bibitem{Raji}
W.~Raji,
Generalized modular forms representable as eta products,
\emph{Acta Arith.} \textbf{129} (2007), no.~1, 41--73.

\bibitem{DiamondShurman}
F.~Diamond and J.~Shurman,
\emph{A First Course in Modular Forms},
Graduate Texts in Mathematics, Vol.~228,
Springer, New York, 2005.

\bibitem{Sturm}
J.~Sturm,
On the congruence of modular forms,
\emph{Invent. Math.} \textbf{79} (1985), 153--160.

\end{thebibliography}
\end{document}